\documentclass[12pt]{amsart}

\usepackage{amsmath}
\usepackage{amssymb}
\usepackage{amscd}
\usepackage{latexsym}
\usepackage{amsthm}
\usepackage{mathrsfs}
\usepackage{color}
\usepackage{setspace}

\newtheorem{thm}{Theorem}[section]
\newtheorem*{thm*}{Theorem}

\newtheorem*{cor*}{Corollary}
\newtheorem{lem}[thm]{Lemma}
\newtheorem*{lem*}{Lemma}
\newtheorem{prop}[thm]{Proposition}
\newtheorem*{prop*}{Proposition}

\theoremstyle{definition}

\newtheorem*{defn*}{Definition}

\theoremstyle{remark}

\newtheorem*{rem*}{Remark}

\newtheorem*{problem*}{Problem}

\newtheorem*{question*}{Question}

\newcommand{\Q}{\mathbb Q}

\newcommand{\CE}{\mathscr E}

\newcommand{\Z}{\mathbb Z}

\DeclareMathOperator{\pExp}{Exp}
\DeclareMathOperator{\End}{End}

\DeclareMathOperator{\Sym}{Sym}

\title{Five-term relation and Macdonald polynomials}
\author{Adriano Garsia}
\address{UC San Diego, 9500 Gilman Drive  \# 0112, La Jolla, CA 92093-0112, USA}
\author{Anton Mellit}
\address{Faculty of Mathematics, University of Vienna, \\ Oskar-Morgenstern-Platz 1, 1090 Vienna, Austria}

\begin{document}
	\onehalfspacing
\begin{abstract}
The non-commutative five-term relation $T_{1,0} T_{0,1} = T_{0,1} T_{1,1} T_{1,0}$
is shown to hold for certain operators acting on symmetric functions. The
``generalized recursion'' conjecture of Bergeron and Haiman is a corollary
of this result.
\end{abstract}

\maketitle

\section{Introduction}
Suppose we have a family of operators $\{R_{m,n}\}$ acting on some vector space $V$, where $m$, $n$
run over all pairs of integers $m,n\geq 0$. Suppose $R_{0,0}$ is the identity operator. For each pair of relatively prime integers $m,n \geq 0$ let
\[
T_{m,n} = \sum_{k=0}^\infty u^{mk} v^{nk} R_{mk, nk} \in \End[V][[u,v]].
\]
We say that the \emph{five-term relations}
hold if for any integers $m,n,m',n'\geq 0$ such that $m n'-m'n=1$ we have:
\begin{equation}\label{eq:fiveterm}
T_{m,n} T_{m',n'} = T_{m',n'} T_{m+m', n+n'} T_{m,n}.
\end{equation}
For instance, for $(m,n)=(1,0)$ and $(m',n')=(0,1)$ we obtain
\begin{equation}\label{eq:fiveterm0}
T_{1,0} T_{0,1} = T_{0,1} T_{1, 1} T_{1,0}.
\end{equation}

The motivation for the name comes from the fact that Faddeev-Kashaev's quantum dilogarithm \cite{faddeev1994quantum} satisfies a similarly-looking identity. In their setup $\hat U$ and $\hat V$ are arbitrary operators satisfying $\hat U \hat V = q \hat V \hat U$, the quantum dilogarithm is 
\[
\Psi(x) = \prod_{n=1}^\infty (1- x q^n),
\]
the identity is
\[
\Psi(\hat V) \Psi(\hat U) = \Psi(\hat U) \Psi(-\hat U \hat V) \Psi(\hat V),
\]
and they show that it is a deformation of the classical five-term relation
\[
L(x) + L(y) - L(xy) = L\left(\frac{x-xy}{1-xy}\right) + L\left(\frac{y-xy}{1-xy}\right)
\]
for the Rogers dilogarithm
\[
L(x) = \log(1-x)\log(x)/2 - \int_0^x \frac{\log(1-z)}{z} dz.
\]

In \cite{kontsevich2008stability} one can find the identity\footnote{The identities \eqref{eq:fiveterm} are also valid there due to the obvious $SL_2(\Z)$ invariance} \eqref{eq:fiveterm0}, where it is related to wall-crossing for Donaldson-Thomas invariants.

Using relations \eqref{eq:fiveterm}, the operator $T_{m,n}$ for any $m,n$ can be expressed as a composition using $T_{0,1}$, $T_{1,0}$ and their inverses. Thus we obtain, in particular, that $R_{m,n}$ for any $m,n\in\Z_{\geq 0}$ belongs to the algebra generated by the elements of the form $R_{0,k}$ and $R_{k,0}$ ($k\in\Z_{>0}$). Then \eqref{eq:fiveterm}, when expressed in terms of these elements, provide some interesting relations between them. For instance, taking the coefficient of $u^{k} v$ in both sides of \eqref{eq:fiveterm0} we obtain
\[
[R_{k,0}, R_{0,1}] = R_{1,1} R_{k-1,0}.
\]
For $k=1$ we have $[R_{1,0}, R_{0,1}] = R_{1,1}$, thus in general
\[
[R_{k,0}, R_{0,1}] = [R_{1,0}, R_{0,1}] R_{k-1,0}.
\]
It is an interesting problem to describe a complete set of relations between the operators $R_{0,k}$, $R_{k,0}$ which is in some sense smaller than the original set implied by \eqref{eq:fiveterm}.

Here we prove that our relations are satisfied by certain operators acting on the space of symmetric functions in infinitely many variables over the field $\Q(q,t)$. To define these operators we need the modified Macdonald polynomials of \cite{garsia1999explicit}. These are symmetric functions $\tilde H_\mu$ with coefficients in $\Z[q,t]$, orthogonal with respect to the modified Hall scalar product defined in the power sum basis by\footnote{Our definition differs from the one in \cite{garsia1999explicit} by a sign $(-1)^{|\lambda|}$. This does not affect the orthogonality statement.}
\[
(p_\lambda, p_\mu)_* = \prod_{i=1}^{l(\lambda)} (-(1-q^{\lambda_i})(1-t^{\lambda_i})) (p_\lambda,p_\mu).
\]
We recall that the (unmodified) Hall scalar product is given by
\[
(p_\lambda,p_\mu) = \begin{cases}
z_\lambda & \text{ if $\lambda=\mu$,}\\
0 & \text{otherwise,}
\end{cases}
\]
where $z_\lambda = 1^{\alpha_1} \alpha_1! 2^{\alpha_2} \alpha_2!\cdots$ for a partition $\lambda=1^{\alpha_1} 2^{\alpha_2}\cdots$.

Now we will construct the operators. In fact, we have two statements. In the first statement we set 
\begin{equation}\label{setup1}
R_{k,0} = h_k^\perp, \quad R_{0,k}=(-1)^k\Delta'_{e_k},
\end{equation}
where the operator $h_k^\perp$ is the operator conjugate to the operator of multiplication by $h_k$ with respect to the Hall scalar product, $\Delta'_{F}$ for a symmetric function $F$ denotes the operator defined in the basis of the modified Macdonald polynomials as follows:
\[
\Delta_{F} \tilde H_{\lambda} = F[B_\lambda] \tilde H_{\lambda},\quad B_\lambda = \sum_{(r,c)\in\lambda} q^c t^r,
\]
\[
\Delta_{F}' = \Delta_{F'}\quad\text{for $F'[X]=F[-1/M+X]$,}
\]
where $M=(1-q)(1-t)$.
In the second statement
\begin{equation}\label{setup2}
R_{k,0}=(-1)^k\Delta'_{e_k}, \quad R_{0,k} = (-1)^k \underline{e}_k\left[\frac{X}{M}\right],
\end{equation}
where $\underline{e}_k\left[\frac{X}{M}\right]$ denotes the operator of multiplication by $e_k\left[\frac{X}{M}\right]$. The two statements are related by the conjugation with respect to the modified Hall scalar product.

\begin{thm}\label{thm:main}
The operators defined by \eqref{setup1} (alternatively, by \eqref{setup2}) extend to a family of operators $\{R_{m,n}\}$ satisfying the five-term relations \eqref{eq:fiveterm}.
\end{thm}

Bergeron and Haiman conjectured (\cite{bergeron2013tableaux}, Conjecture 6) certain identities between the $\Delta_{e_k}$ and $h_k^\perp$ operators, which imply interesting recursion relations for the Macdonald polynomials. We show that their statement follows by expanding \eqref{eq:fiveterm0}. In fact, it was our attempt to prove Conjecture 6 that led us to the discoveries of the present work. 

\section{The proof}
We introduce extra variables $u$ and $v$ and set\footnote{We apologize for using conflicting notations $\Delta_F$ and $\Delta_v$ and hope this does not cause confusion.}
\[
\tau_u = \sum_{n=0}^\infty u^n h_n^{\perp},
\quad
\Delta_v = \sum_{n=0}^\infty (-v)^n \Delta_{e_n},
\quad
\Delta_v' = \sum_{n=0}^\infty (-v)^n \Delta'_{e_n} = \pExp[v/M] \Delta_v.
\]
Then we have, for any symmetric function $F$,
\begin{align}\label{eq:delta definition}
(\tau_u F)[X] = F[X+u],
\qquad \Delta_v \tilde H_\lambda = \prod_{(r,c)\in \lambda} (1- v q^c t^r) \tilde H_\lambda,
\\ \nonumber
\Delta_v^{-1} = \sum_{n=0}^\infty v^n \Delta_{h_n},\qquad \Delta_v^{-1} \tilde H_\lambda = \prod_{(r,c)\in \lambda} \frac{1}{1- v q^c t^r} \tilde H_\lambda.
\end{align}
It is convenient to modify the Bergeron-Garsia operator $\nabla$ by adding a sign, so that it becomes closer to $\Delta_v$ in shape:
\begin{equation}\label{eq:nabla definition}
\nabla \tilde H_\lambda = \prod_{(r,c)\in \lambda} (-q^c t^r) \tilde H_\lambda = (-1)^{|\lambda|} q^{n'(\lambda)} t^{n(\lambda)} \tilde H_\lambda.
\end{equation}

Now we write the identity (77) of \cite{bergeron2013tableaux}, which is equivalent to Conjecture 6 there, in a generating function form. It claims
\[
\nabla^{-1} h_k^\perp \nabla h_l^\perp = \sum_{r=0}^k (-1)^{k-r} \Delta_{h_r} h_{k+l}^\perp \Delta_{e_{k-r}}\quad (k,l\in\Z_{\geq 0}).
\]
Remember that our $\nabla$ has a sign that conveniently turns $(-1)^r$ to $(-1)^{k-r}$. Now we multiply both sides by $u^k v^l$ and sum up:
\[
\nabla^{-1} \tau_u \nabla \tau_v = \sum_{k,l,r\geq 0}^{r\leq k} (-1)^{k-r} \Delta_{h_r} h_{k+l}^\perp \Delta_{e_{k-r}} u^k v^l.
\]
We change the indexing: $i=k+l$, $j=k-r$, so that the new summation runs over $i,j,r\geq 0$, $i\geq j+r$. We obtain (the old indexes are $k=j+r$, $l=i-(j+r)$):
\[
\sum_{j,r\geq 0, i\geq j+r} (-1)^j \Delta_{h_r} h_i^\perp \Delta_{e_j} u^{j+r} v^{i-(j+r)} = \Delta_{u/v}^{-1} \tau_v \Delta_{u/v}\; \Big|_{v^{\geq 0}},
\]
where the notation $\Big|_{v^{\geq 0}}$ stands for ``keep only the terms with non-negative power of $v$''.
\begin{prop}[Generating function form of the conjecture] Conjecture 6 of \cite{bergeron2013tableaux} is equivalent to
\begin{equation}\label{eq:1}
\nabla^{-1} \tau_u \nabla \tau_v = \Delta_{u/v}^{-1} \tau_v \Delta_{u/v}\; \Big|_{v^{\geq 0}}
\end{equation}
\end{prop}

Now we have the following observation:
\begin{prop}\label{prop:poly1}
For a symmetric function $F$ of degree $d$ the operator
\[
\Delta_v^{-1} F^\perp \Delta_v
\]
is a polynomial in $v$ of degree $\leq d$. The coefficient of $v^d$ is $\nabla^{-1} F^\perp \nabla$.
\end{prop}
\begin{proof}
	 Pieri rules for Macdonald polynomials (VI.6, \cite{macdonald1995symmetric}) say that
\[
\tilde H_{1^r} \tilde H_\lambda = \sum_{\mu} B_{\mu,\lambda}(q,t) \tilde H_\mu
\]
for certain rational functions $B_{\mu,\lambda}(q,t)$, where the sum is over the partitions $\mu\supset\lambda$ such that $\mu\setminus\lambda$ is a vertical $r$-strip. Since the functions $\tilde H_{1^r}$ generate the ring of symmetric functions, we obtain that for any symmetric function $F$ of degree $d$
\[
F \tilde H_\lambda = \sum_{\mu\supset\lambda} C^F(q,t)_{\mu,\lambda} \tilde H_\mu
\]
for some rational functions $C^F(q,t)_{\mu,\lambda}$, where the summation is over $\mu\supset\lambda$ such that $|\mu\setminus\lambda|=d$. Dualizing, we obtain that for any symmetric function $F$ of degree $d$
\[
F^\perp \tilde H_\lambda = \sum_{\mu\subset\lambda} C^{F\perp}(q,t)_{\mu,\lambda} \tilde H_\mu
\]
for some rational functions $C^{F\perp}(q,t)_{\mu,\lambda}$, where the summation is over $\mu\subset\lambda$ such that $|\lambda\setminus\mu|=d$. Applying \eqref{eq:delta definition}, we obtain
\[
\Delta_v^{-1} F^\perp \Delta_v \tilde H_\lambda = \sum_{\mu\subset\lambda} C^{F\perp}(q,t)_{\mu,\lambda} \tilde H_\mu \prod_{(r,c)\in\lambda\setminus\mu} (1-v q^c t^r).
\]
The right hand side is a polynomial of degree $\leq d$ and the coefficient of $v^d$ is
\[
\sum_{\mu\subset\lambda} C^{F\perp}(q,t)_{\mu,\lambda} \tilde H_\mu \prod_{(r,c)\in\lambda\setminus\mu} (- q^c t^r) = \nabla^{-1} F^\perp \nabla \tilde H_\lambda
\]
by \eqref{eq:nabla definition}.
\end{proof}

Proposition \ref{prop:poly1} implies that in $\Delta_{u/v}^{-1} \tau_v \Delta_{u/v}$ the exponent of $u/v$ is less or equal to the exponent of $v$. Therefore we don't have any terms with negative powers of $v$ and we can omit $\Big|_{v^{\geq 0}}$ from (\ref{eq:1}). So the conjecture is equivalent to the following:
\[
\nabla^{-1} \tau_u \nabla \tau_v = \Delta_{u/v}^{-1} \tau_v \Delta_{u/v}.
\]
Now we move $\tau_v$ to the right, substitute $uv$ in $u$ and interchange $u$ and $v$:
\[
\nabla^{-1} \tau_{uv} \nabla = \Delta_{v}^{-1} \tau_u \Delta_{v} \tau_u^{-1}.
\]
This is what we are going to prove. We write the right hand side as follows:
\begin{equation}\label{eq:2}
\sum_{i,j\geq 0} W_{i,j}u^i v^j:=\Delta_{v}^{-1} \tau_u \Delta_{v} \tau_u^{-1}
\end{equation}
The main idea is to put parentheses in the right hand side in two different ways.

\subsection{First way} $(\Delta_{v}^{-1} \tau_u \Delta_{v}) \tau_u^{-1}$. By Proposition
\ref{prop:poly1}, we know that in each non-zero term the exponent of $v$ is less than or equal to the exponent of $u$. Multiplying by $\tau_u^{-1}$ only increases the exponents of $u$. Moreover, to calculate the terms where the exponent of $v$ equals the exponent of $u$ we can replace $\Delta$ by $\nabla$. Therefore we have
\begin{prop}\label{prop:firstway}
We have $W_{i,j}=0$ for $i<j$ and $W_{i,i} = \nabla^{-1} h_i^\perp \nabla$.
\end{prop}

\subsection{Second way} $\Delta_{v}^{-1} (\tau_u \Delta_{v} \tau_u^{-1})$. It turns out that there is exactly the same statement about $\tau_u \Delta_{v} \tau_u^{-1}$ as we had about $\Delta_{v}^{-1} \tau_u \Delta_{v}$. To proceed we need to introduce a (partially defined) operator $S^{-1}$ which acts on operators. Set $\tau=\tau_1$, $M=(1-q)(1-t)$,
\[
\tau F = F[X+1],\quad \tau^* F = \pExp\left[-\frac{X}{M}\right] F,
\]
where $\pExp$ is the plethystic exponential,
\[
\pExp[X] = \sum_{n=0}^\infty h_n[X] = \exp\left( \sum_{n=1}^\infty \frac{p_n[X]}n\right).
\]
The operator $\tau$ sends symmetric functions to symmetric functions, but the operator $\tau^*$ sends a symmetric function to an infinite series so that the degrees of the terms tend to infinity.

Denote by $\Sym[[X]]$ the algebra of infinite series of the form
\begin{equation}\label{eq:formal series}
F=F_0+F_1+F_2+\cdots,
\end{equation}
where $F_i$ is a symmetric function of degree $i$ for each $i$. It is convenient to think of $\Sym[[X]]$ as a complete topological algebra. The topology is defined in such a way that a basis of neighborhoods of $0$ is given by the sets
\[
\Sym^{\geq d}[[X]]:=\{F_0+F_1+F_2+\cdots\,|\, F_i=0\;\text{for all $i<d$}\}.
\]
Then a sequence of infinite series $F^{(1)}, F^{(2)},\ldots$ is a Cauchy sequence if for each $d\geq 0$ the sequence of degree $d$ terms $F^{(1)}_d, F^{(2)}_d,\ldots$ eventually stabilizes. Each Cauchy sequence clearly has a limit, therefore $\Sym[[X]]$ is complete. Any formal series can be approximated by polynomials, for instance for $F$ as in \eqref{eq:formal series}, we have
\[
F = \lim_{d\to\infty} F_{\leq d},\qquad F_{\leq d}:=\sum_{i=0}^d F_i\in\Sym[X].
\]
Therefore $\Sym[X]$ is dense in $\Sym[[X]]$. The topology on $\Sym[X]$ induced from $\Sym[[X]]$ has as a basis of neighborhoods of $0$ the sets $\Sym^{\geq d}[X]=\Sym^{\geq d}[[X]]\cap \Sym[X]$. Then $\Sym[[X]]$ is the completion of $\Sym[X]$. In particular, this implies that any continuous linear operator $\Sym[X]\to\Sym[[X]]$ can be uniqely extended to a continous linear operator $\Sym[[X]]\to\Sym[[X]]$.

From this point of view, the operator $\tau^*:\Sym[X]\to \Sym[[X]]$ uniquely extends to a continuous operator $\Sym[[X]]\to\Sym[[X]]$, which we denote by the same symbol $\tau^*$. On the other hand, $\tau$ is not continuous, and it does not extend to a continuous operator $\Sym[[X]]\to\Sym[[X]]$. The composition $\tau^*\tau$ is a linear operator $\Sym[X]\to\Sym[[X]]$ which is not continuous.

\begin{prop}\label{prop:S}
Let $L:\Sym[X]\to \Sym[X]$ be a continuous linear operator. There exists at most one continuous linear operator\footnote{The operation $S^{-1}$, after a sign change, turns out to be the inverse of the operation $S$ from \cite{bergeron2016some}, \cite{bergeron2015compositional}} $S^{-1}(L):\Sym[X]\to \Sym[X]$ such that for all $F\in \Sym[X]$
\begin{equation}\label{eq:S identity}
\tau^*\tau L F = S^{-1}(L) \tau^*\tau F.
\end{equation}
The set of all continuous linear operators $L:\Sym[X]\to \Sym[X]$ such that $S^{-1}(L)$ exists forms an algebra, and the operation $S^{-1}$ is an algebra homomorphism.
\end{prop}
\begin{proof}
Note that on the left hand side of \eqref{eq:S identity} we have $L F\in\Sym[X]$, $\tau L F\in\Sym[X]$ and $\tau^*\tau L F\in\Sym[[X]]$. On the right hand side we have $\tau F\in\Sym[X]$, $\tau^* \tau F\in\Sym[[X]]$, $S^{-1}(L) \tau^*\tau F\in\Sym[[X]]$.

First we prove the uniqueness. Suppose we have two continuous linear operators $L'$, $L''$ such that
\[
\tau^*\tau L F = L' \tau^*\tau F,\quad \tau^*\tau L F = L'' \tau^*\tau F\qquad (F\in\Sym[X]).
\]
Then the difference $L''-L'$ vanishes on the set $\tau^*\tau\Sym[X]$. Since $\tau:\Sym[X]\to\Sym[X]$ is invertible, we have $\tau\Sym[X]=\Sym[X]$. Therefore $L''-L'$ vanishes on the set $\tau^*\Sym[X]$. By continuity, we have that $(L''-L')\tau^*$ vanishes on $\Sym[[X]]$, so $(L''-L')\tau^*=0$. Since $\tau^*:\Sym[[X]]\to\Sym[[X]]$ is invertible (the inverse is given by the operator of multiplication by $\pExp\left[\frac{X}{M}\right]$), we have $L''=L'$.

Now suppose $L_1,L_2$ are such that $S^{-1}(L_1)$, $S^{-1}(L_2)$ exist. Then for any $F\in\Sym[X]$ we have
\[
\tau^*\tau(L_1+L_2) F = (S^{-1}(L_1)+S^{-1}(L_2))\tau^*\tau F,
\]
\[
\tau^*\tau L_1 L_2 F = S^{-1}(L_1) \tau^* \tau L_2 F = S^{-1}(L_1) S^{-1}(L_2) \tau^* \tau F,
\]
which shows that $S^{-1}(L_1+L_2)$ exists and is given by $S^{-1}(L_1)+S^{-1}(L_2)$, and similarly $S^{-1}(L_1L_2)$ exists and is given by $S^{-1}(L_1)S^{-1}(L_2)$. This completes the proof.
\end{proof}

Operators satisfying the conditions of Proposition \ref{prop:S} can be built up from the operators $D_n$.
\begin{prop}\label{prop:S of Dn}
For any $n\in \Z$, define an operator $D_n:\Sym[X]\to \Sym[X]$ by 
\[
D_n = F[X+M z^{-1}] \pExp[-Xz]\Big|_{z^n}.
\]
Then $D_n$ is a continuous linear operator and $S^{-1}(D_n)=-D_{n-1}$.
\end{prop}
\begin{proof}
	The operator $D_n$ is homogeneous of degree $n$ in the sense that the degree of $D_n F$ is $d+n$ for any $F\in\Sym[X]$ of degree $d$. Therefore $D_n$ is continuous. Let $F\in\Sym[X]$. Expanding the definitions and using $\pExp[-z]=1-z$, we obtain
	\[
	\tau^* \tau D_n F = F[X+M z^{-1}+1] \pExp\left[-Xz-z-\frac{X}M\right]\Big|_{z^n}
	\]
	\[
	=F[X+M z^{-1}+1] \pExp\left[-Xz-\frac{X}M\right]\Big|_{z^n} - F[X+M z^{-1}+1] \pExp\left[-Xz-\frac{X}M\right]\Big|_{z^{n-1}}.
	\]
	Similarly, but this time using $\pExp[-z^{-1}]=1-z^{-1}$, we obtain
	\[
	-D_{n-1} \tau^* \tau F = - F[X+M z^{-1}+1] \pExp\left[-Xz-z^{-1}-\frac{X}M\right]\Big|_{z^{n-1}}
	\]
	\[
	=- F[X+M z^{-1}+1] \pExp\left[-Xz-\frac{X}M\right]\Big|_{z^{n-1}} + F[X+M z^{-1}+1] \pExp\left[-Xz-\frac{X}M\right]\Big|_{z^n}.
	\]
	We see that $S^{-1}(D_n)$ exists and equals $-D_{n-1}$.
\end{proof}

Now we can formulate a statement analogous to Proposition \ref{prop:poly1}.
\begin{prop}\label{prop:poly2}
For a symmetric function $F$ of degree $d$, the operator
\[
\tau_u \Delta_{F} \tau_u^{-1}
\]
is a polynomial in $u$ of degree $\leq d$. The coefficient of $u^d$ is the operator $S^{-1}(\Delta_{F}')$.
\end{prop}
\begin{proof}
We note that we can replace $\Delta_F$ by $\Delta_{F}'$ first. Indeed, this does not affect the first statement, and assuming the first statement, the coefficients of $u^d$ in $\tau_u \Delta_{F} \tau_u^{-1}$ and $\tau_u \Delta_{F}' \tau_u^{-1}$ are the same because $F[X]-F[-1/M+X]$ is a sum of functions of degrees less than $d$.

The proof hinges on the fact that the operator $\Delta_{F}'$ can be written as a linear combination of the operators $D_{i_1} D_{i_2} \cdots D_{i_d}$ (Lemma \ref{lem:delta expansion}). This is well-known, but unfortunately we could not find a complete reference for this fact, so we provide a proof in Section \ref{sec:delta expansion}.

Consider $\tau_u D_n \tau_u^{-1}$:
\[
F[X]\to F[X-u] \to F[X+M z^{-1} - u] \pExp[-Xz]\Big|_{z^n} \to F[X+M z^{-1}] \pExp[-Xz-uz]\Big|_{z^n}.
\]
Because $\pExp[-uz]=1-uz$, we obtain
\[
\tau_u D_n \tau_u^{-1} = D_n - u D_{n-1}.
\]
So we see that we obtained a polynomial in $u$ of degree $\leq 1$, and the top coefficient is $-D_{n-1}=S^{-1}(D_n)$. The claim follows from this and the expansion of $\Delta_F'$ as a linear combination of the operators $D_{i_1} D_{i_2} \cdots D_{i_d}$.
\end{proof}

Proposition \ref{prop:poly2} implies that in each term of $\tau_u \Delta_{v} \tau_u^{-1}$ the exponent of $u$ is less than or equal to the exponent of $v$. Multiplication by $\Delta_v^{-1}$ on the left only increases the exponent of $v$. This proves the following
\begin{prop}\label{prop:secondway}
We have $W_{i,j}=0$ for $i>j$ and $W_{i,i} = S^{-1}(\Delta_{e_i}')$.
\end{prop}

Putting Propositions \ref{prop:firstway} and \ref{prop:secondway} together we obtain
\begin{thm}
The commutator of $\Delta_v^{-1}$ and $\tau_u$ is a powers series in $uv$ and it has the following two expressions:
$$
\Delta_{v}^{-1} \tau_u \Delta_{v} \tau_u^{-1} = \nabla^{-1} \tau_{uv} \nabla
= S^{-1}(\Delta_{uv}').
$$
\end{thm}

This establishes Conjecture 6 of \cite{bergeron2013tableaux}. To complete our proof of Theorem \ref{thm:main}, denote for any operator $L$
\[
N(L) = \nabla L \nabla^{-1},\quad N^{-1}(L) = \nabla^{-1} L \nabla.
\]
Let also $N^{-1}$ send $u, v$ to $uv, v$ respectively, and let $S^{-1}$ send $u,v$ to $u, uv$ respectively. 
We define $R_{k,0}$, $R_{0,k}$ and then $T_{0,1}$, $T_{1,0}$ as in \eqref{setup1}, so that
\[
T_{1,0} = \tau_u,\quad T_{0,1} = \Delta_v',
\]
\begin{equation}\label{eq:id0}
T_{0,1}^{-1} T_{1,0} T_{0,1} T_{1,0}^{-1} = N^{-1}(T_{1,0}) = S^{-1}(T_{0,1}).
\end{equation}
Hence we must necessarily have
\[
T_{1,1} = N^{-1}(T_{1,0}) = S^{-1}(T_{0,1}).
\]
Moreover, we have
\begin{equation}\label{eq:glue}
N^{-1}(T_{0,1}) = T_{0,1},\quad S^{-1}(T_{1,0}) = \pExp\left[\frac{u}M\right] T_{1,0}.
\end{equation}
From this data there is a unique way to construct operators $T_{m,n}$ for all relatively prime pairs $m,n$ such that
\[
N^{-1}(T_{m,n}) = T_{m, m+n},\quad S^{-1}(T_{m,n}) = T_{m+n, n}\; (n>0).
\]
Moreover, applying the operators $N^{-1}$, $S^{-1}$ to \eqref{eq:id0} we obtain the statements \eqref{eq:fiveterm} for all $m,n,m',n'$ with $mn'-m'n=1$. 

The statement for $R_{k,0}$, $R_{0,k}$ as in \eqref{setup2} is obtained from the statement we just established by applying the conjugation with respect to the modified Hall scalar product. Modified Macdonald polynomials are orthogonal with respect to the modified Hall scalar product, so the operators $\Delta_F$, $\Delta_F'$ are self-adjoint. The adjoint of $h_k^\perp$ is given by $(-1)^k \underline e_k\left[\frac{X}{M}\right]$.

\section{Expansion of $\Delta_F'$ in terms of $D_n$}\label{sec:delta expansion}

\begin{lem}\label{lem:delta expansion}
	For any $F\in \Sym[X]$ of degree $d$, the operator $\Delta_F'$ can be written as a finite sum of the form
	\[
	\Delta_F' = \sum_{i_1,\ldots,i_d} c_{i_1,\ldots,i_d} D_{i_1} D_{i_2} \cdots D_{i_d},
	\]
	where each $c_{i_1,\ldots,i_d}$ is a rational function of $q$ and $t$.
\end{lem}
\begin{proof}
	Let $\CE_d$ denote the space of operators $\Sym\to \Sym$ that can be written as linear combinations of operators of the form $D_{i_1} D_{i_2} \cdots D_{i_d}$ with coefficients in $\Q(q,t)$, $i_j\in\Z$. So the statement of Lemma is $\Delta_F'\in\CE_d$.
	
	By Theorem 1.3 of \cite{bergeron2016some},
	\begin{equation}\label{eq:operator h_n}
	\frac{qt}{qt-1} \nabla \underline h_n\left[X\left(\frac{1}{qt}-1\right)\right] \nabla^{-1} = \frac{1}{M} [\Phi_b, \Psi_a],
	\end{equation}
	where $a,b$ are positive integers satisfying $a+b=n$. Using the recursions of Theorem 1.1 of \cite{bergeron2016some} together with the initial values $\Phi_1=\frac{1}{M}[D_1,D_0]$, $\Psi_2=D_2$, we show that $\Phi_b\in\CE_b$ for $b\geq 1$ and $\Psi_a\in\CE_a$ for $a\geq 2$. Thus, for $n\geq 3$ the operator \eqref{eq:operator h_n} is in $\CE_n$. For $n=2$, we have (using the definition and the Jacobi identity):
	\[
	\frac{1}{M} [\Phi_1, \Psi_1] = \frac{1}{M^2} [[D_1,D_0], -\underline e_1] = -\frac{1}{M^2}[[D_1,\underline e_1],D_0]-\frac{1}{M^2}[D_1,[D_0,\underline e_1]].
	\]
	From the well-known identity $\frac{1}{M}[D_n,\underline e_1]=D_{n+1}$, we deduce
	\[
	\frac{1}{M} [\Phi_1, \Psi_1] = -\frac{1}{M}[D_2,D_0] \in\CE_2.
	\]
	For $n=1$, we have
	\[
	\frac{qt}{qt-1} \nabla \underline h_1\left[X\left(\frac{1}{qt}-1\right)\right] \nabla^{-1} = -\nabla \underline h_1 \nabla^{-1} = D_1 \in\CE_1.
	\]
	So for all $n\geq 1$ the operator in the left hand side of \eqref{eq:operator h_n} is in $\CE_n$.
Expressing any symmetric function as a polynomial in the functions $h_n\left[X\left(\frac{1}{qt}-1\right)\right]$,
we see that $\nabla \underline G \nabla^{-1}\in \CE_d$ for any symmetric function $G$ of degree $d$.

The operation $S^{-1}$ is an algebra homomorphism by Proposition \ref{prop:S}. So Proposition \ref{prop:S of Dn} implies that $S^{-1}$ preserves $\CE_d$. By Lemma \ref{lem:S delta} below,
\[
\Delta_{F}' = S^{-1}\left(\nabla \underline{G} \nabla^{-1}\right),
\]
where $G[X]=F\left[\frac{X}{M}\right]$. Hence $\Delta_{F}'\in\CE_d$.
\end{proof}

\begin{lem}\label{lem:S delta}
	For any symmetric function $F$, the operator $S^{-1}(\nabla \underline F \nabla^{-1})$ exists and equals $\Delta_{F[MX]}'$.
\end{lem}
\begin{proof}
We need to verify that for any $F,G\in\Sym[X]$
\[
\tau^*\tau \nabla \left(F \nabla^{-1} G\right) = \Delta_{F[MX]}' \tau^* \tau G.
\]
Equivalently, for any $F,G\in\Sym[X]$
\begin{equation}\label{eq:delta S identity}
\tau^*\tau \nabla (F G) = \Delta_{F[MX]}' \tau^* \tau \nabla G.
\end{equation}
Let us prove this identity for $G=1$ first. In this case, the statement is
\begin{equation}\label{eq:delta S identity simplified}
\tau^*\tau \nabla F = \Delta_{F[MX]}' \pExp\left[-\frac{X}{M}\right].
\end{equation}
We will use the Cauchy formula
\begin{equation}\label{eq:cauchy}
\pExp\left[-\frac{XY}{M}\right] = \sum_{\lambda} \frac{\tilde H_\lambda[X] \tilde H_\lambda[Y]}{(\tilde H_\lambda, \tilde H_\lambda)_*}.
\end{equation}
The right hand side of \eqref{eq:delta S identity simplified} has the following expansion in the modified Macdonald basis:
\[
\sum_{\lambda}  \frac{F[M B_\lambda-1] \tilde H_\lambda[X]}{(\tilde H_\lambda, \tilde H_\lambda)_*}.
\]
So in order to verify that the left hand side has the same expansion, we need to show that for each partition $\lambda$ we have
\[
(\tau^* \tau \nabla F, \tilde H_\lambda)_* = F[M B_\lambda-1].
\]
Since $\nabla$ is self-adjoint and $\tau^*$ is the adjoint operator of $\tau$, we have
\[
(\tau^* \tau \nabla F, \tilde H_\lambda)_* = (F, \nabla\tau^*\tau \tilde H_\lambda)_*.
\]
Theorem I.3 from \cite{garsia1999explicit} allows us to evaluate $\nabla\tau^*\tau \tilde H_\lambda$, and we obtain
\[
\left(F, \nabla\tau^*\tau \tilde H_\lambda\right)_* = \left(F, \pExp\left[-\frac{X(M B_\lambda-1)}{M}\right]\right)_*.
\]
Applying \eqref{eq:cauchy} for $Y=M B_\lambda-1$, we obtain
\[
\left(F, \nabla\tau^*\tau \tilde H_\lambda\right)_* = \sum_{\lambda} \frac{(F,\tilde H_\lambda)_*\; \tilde H_\lambda[M B_\lambda-1]}{(\tilde H_\lambda, \tilde H_\lambda)_*} = F[M B_\lambda-1].
\]
So \eqref{eq:delta S identity simplified} has been established. Now suppose $G$ is arbitrary. Using \eqref{eq:delta S identity simplified} for the product $FG$ instead of $F$, we write the left hand side of \eqref{eq:delta S identity} as
\[
\Delta_{F[M X] G[M X]}' \pExp\left[-\frac{X}{M}\right].
\]
Applying \eqref{eq:delta S identity simplified} for $G$ instead of $F$, we write the right hand side of \eqref{eq:delta S identity} as
\[
\Delta_{F[M X]}' \Delta_{G[M X]}' \pExp\left[-\frac{X}{M}\right].
\]
So \eqref{eq:delta S identity} follows from the equality of the operators $\Delta_{F[M X] G[M X]}'$ and $\Delta_{F[M X]}' \Delta_{G[M X]}'$, which is obvious from the definition.
\end{proof}

\section*{Acknowledgments}
The first  author  is  grateful to  Fran\c{c}ois Bergeron and Mark Haiman for communicating their conjectures  in the earliest stages of their  discoveries.

The first author's work is supported by NSF grant 1362160.

The second author is grateful to SISSA, Trieste where he worked on the first version of the present paper. The second author's research is supported by the Austrian Science Fund (FWF) through
the START-Project Y963-N35 of Michael Eichmair.

We are grateful to the anonymous referees for suggestions on improving the text and correcting misprints. 

\bibliographystyle{amsalpha} 
\bibliography{refs.bib}

\providecommand{\bysame}{\leavevmode\hbox to3em{\hrulefill}\thinspace}
\providecommand{\MR}{\relax\ifhmode\unskip\space\fi MR }
\providecommand{\MRhref}[2]{%
  \href{http://www.ams.org/mathscinet-getitem?mr=#1}{#2}
}
\providecommand{\href}[2]{#2}
\begin{thebibliography}{BGLX16}

\bibitem[BGLX15]{bergeron2015compositional}
Francois Bergeron, Adriano Garsia, Emily~Sergel Leven, and Guoce Xin,
  \emph{Compositional (km, kn)-shuffle conjectures}, International Mathematics
  Research Notices \textbf{2016} (2015), no.~14, 4229--4270.

\bibitem[BGLX16]{bergeron2016some}
\bysame, \emph{Some remarkable new plethystic operators in the theory of
  {M}acdonald polynomials}, J. Comb. \textbf{7} (2016), no.~4, 671--714.
  \MR{3538159}

\bibitem[BH13]{bergeron2013tableaux}
Fran{\c{c}}ois Bergeron and Mark Haiman, \emph{Tableaux formulas for
  {M}acdonald polynomials}, International Journal of Algebra and Computation
  \textbf{23} (2013), no.~04, 833--852.

\bibitem[FK94]{faddeev1994quantum}
Ludwig~D Faddeev and Rinat~M Kashaev, \emph{Quantum dilogarithm}, Modern
  Physics Letters A \textbf{9} (1994), no.~05, 427--434.

\bibitem[GHT99]{garsia1999explicit}
A.~M. Garsia, M.~Haiman, and G.~Tesler, \emph{Explicit plethystic formulas for
  {M}acdonald {$q,t$}-{K}ostka coefficients}, S\'em. Lothar. Combin.
  \textbf{42} (1999), Art. B42m, 45, The Andrews Festschrift (Maratea, 1998).
  \MR{1701592}

\bibitem[KS08]{kontsevich2008stability}
Maxim Kontsevich and Yan Soibelman, \emph{Stability structures, motivic
  donaldson-thomas invariants and cluster transformations}, arXiv preprint
  arXiv:0811.2435 (2008).

\bibitem[Mac95]{macdonald1995symmetric}
I.~G. Macdonald, \emph{Symmetric functions and {H}all polynomials}, second ed.,
  Oxford Mathematical Monographs, The Clarendon Press, Oxford University Press,
  New York, 1995, With contributions by A. Zelevinsky, Oxford Science
  Publications. \MR{1354144}

\end{thebibliography}

\end{document}